\documentclass[sn-mathphys,Numbered]{sn-jnl}

\usepackage{graphicx}%
\usepackage{multirow}%
\usepackage{amsmath,amssymb,amsfonts}%
\usepackage{amsthm}%
\usepackage{mathrsfs}%
\usepackage[title]{appendix}%
\usepackage{xcolor}%
\usepackage{textcomp}%
\usepackage{manyfoot}%
\usepackage{booktabs}%
\usepackage{algorithm}%
\usepackage{algorithmicx}%
\usepackage{algpseudocode}%
\usepackage{listings}%

\theoremstyle{thmstyleone}%
\newtheorem{theorem}{Theorem}
\newtheorem{lemma}[theorem]{Lemma}%
\newtheorem{proposition}[theorem]{Proposition}%
\newtheorem{claim}[theorem]{Claim}%

\theoremstyle{thmstyletwo}%

\theoremstyle{thmstylethree}%
\newtheorem{definition}{Definition}%

\def\bb{{\mathcal B}} 
\def\rr{{\mathcal R}}
\def\GG{{\mathcal G}}
\def\ii{{\mathcal I}}
\newcommand{\ka}{\kappa}
\usepackage{mathtools} 
\DeclarePairedDelimiter\ceil{\lceil}{\rceil}
\DeclarePairedDelimiter\floor{\lfloor}{\rfloor}
\begin{document}

\title[Independence numbers of the $2$-token graphs...]{Independence numbers of the $2$-token graphs of some join graphs}


\author*[1]{\fnm{Luis Manuel} \sur{Rivera}}\email{luismanuel.rivera@gmail.com}

\author[2]{\fnm{Gerardo} \sur{Vazquez Briones}}\email{vabge0310@hotmail.com}

\affil[1,2]{\orgdiv{Unidad Acad\'emica de Matem\'aticas}, \orgname{Universidad Aut\'onoma de Zacatecas}, \orgaddress{\street{Calzada
Solidaridad entronque Paseo a la Bufa}, \city{Zacatecas}, \postcode{98000}, \state{Zacatecas}, \country{Mexico}}}


\abstract{The \textit{$2$-token graph} $F_2(G)$ of a graph $G$ is the graph whose set of vertices consists of all the $2$-subsets of $V(G)$, where two vertices are adjacent if and only if their symmetric difference is an edge in $G$. Let $G$ be the join graph of $E_n$ and $H$, where $H$ is any graph. In this paper, we give a method to construct an independent set $\ii'$ of $F_2(G)$ from an independent set $\ii$ of $F_2(G)$ such that $|\ii'| \geq |\ii|$. As an application, we obtain the independence number of the $2$-token graphs of fan graphs $F_{n, m}$, wheel graphs $W_{n, m}$ and $E_n+K_n$. 
}




\maketitle

\section{Introduction}
The \textit{$k$-token graph} $F_k(G)$ of a graph $G$ is the graph whose set of vertices consists of all the $k$-subsets of $V(G)$, where two vertices are adjacent if and only if their symmetric difference is an edge in $G$. It seems that this class of graphs was defined by Johns~\cite{johns}, under the name of $n$-subgraph graph of a graph. In 1991 Alavi et al.,~\cite{alavi1}, independently, reintroduced the $2$-token graphs calling them \textit{double vertex graphs} and together with his co-authors studied several of their properties (see e.g. \cite{alavi2, alavi3}). For several years, the study of this class of graphs apparently stopped. However, in 2002, Rudolph~\cite{rudolph}, redefined the token graphs with the name of \textit{symmetric powers of graphs} \cite{aude}, to study the graph isomorphism problem and as a tool to model a phenomenon in quantum mechanics (see e.g. \cite{aude, barghi, fisch} for other related works) Later, in 2012 Fabila-Monroy et al.,~\cite{FFHH} reintroduced the concept of $k$-token graphs as ``a model in which $k$ indistinguishable tokens move from vertex to vertex along the edges of a graph". From this article on, the study of the properties of token graphs has increased significantly, see e.g. ~\cite{adame2021hamiltonicity, fabila2022connectivity, gomez, leanos2021edge,  reyes2023spectra, zhang2023automorphisms}.

In this paper, we are interested in the independence number of token graphs. Alba et al.,~\cite{alba} began the study of the independence number of $k$-token graphs, in particular, they obtained the exact independence number of the $2$-token graph of path graphs, cycle graphs, complete bipartite graphs among others. Deepalakshmi et al.,~\cite{deepa} presented a sharp lower bound on $\alpha(F_k(G))$. Abdelmalek et al.,~\cite{abdelmalek}, studied the well-covered token graphs (when all of the maximal independent sets of $F_k(G)$ have the same cardinality). In particular, they showed a characterization of the well-covered $F_k(G)$ when $G$ is a connected bipartite graph. Finally, J\'imenez-Sep\'ulveda and Rivera~\cite{jimenez} obtained the independence number of the $2$-token graphs of the fan graph $F_{1, m}$ and of the wheel graph $W_{1, m}$.  

\subsection{Main results}
In this paper, we obtain the independence number of the $2$-token graphs of $F_{n, m}$ (Theorem~\ref{tabanico}) and $W_{n, m}$ (Theorem~\ref{trueda}), for $n \geq 1$, which generalizes the results obtained by the first author and J\'imenez-Sep\'ulveda~\cite{jimenez}. These graphs are a particular case of a graph product. The {\it join graph} $G_1+G_2$ of two disjoint graphs $G_1$ and $G_2$ is defined as the graph whose vertex set is $V(G_1)\cup V(G_2)$ and its edge set is $E(G)=E(G_1) \cup E(G_2) \cup \{uv \colon u \in V(G_1) \text{  and } v \in V(G_2)\}$. Some special cases are the bipartite complete graph $K_{n,m}=E_n+E_m$, the fan graphs $F_{n, m}=E_n+P_m$ and the wheel graphs $W_{n,m}=E_n+C_m$, where $E_n=\overline{K_n}$.

Let $H$ be any graph and $G=E_n+H$. In this paper, we show a method to construct an independent set $\ii'$ of $F_2(G)$ from a given independent set $\ii$ of $F_2(G)$ such that $|\ii'|\geq |\ii|$ (Lemma~\ref{conjasociado}). Using this result, we obtain, in particular, the independence number of the $2$-token graph of $F_{n, m}$, $W_{n, m}$ and $E_n+K_n$ (Theorem~\ref{teokmen}), but we believe that this technique can be applied for other graphs of the form $E_n+H$. Another result, that is used in the proof of Theorems~\ref{tabanico} and~\ref{trueda}, is the independence number of the $2$-token graph of $P_m-A$, where $A \subset V(P_m)$ (Theorem~\ref{propcaminomenos}). 

The outline of the paper is as follows. In Section~\ref{sec2} we present some preliminary results and the proof of Theorem~\ref{propcaminomenos}. In Sections~\ref{sec3},~\ref{sec4},~\ref{sec5}, we present the proofs of Theorems~\ref{tabanico},~\ref{trueda}, and~\ref{teokmen}, respectively. Finally, in Section~\ref{sec6} we use Lemma~\ref{conjasociado} to obtain $\alpha(F_2(K_{n,m}))$ which provides a simpler proof than the one presented by de Alba et. al. \cite{alba}.

\section{Preliminary results}\label{sec2}
First, we give some definitions and notation which will be needed throughout the paper. An {\it odd component} $G'$ of a graph $G$ is a component of $G$ such that $|G'|$ is odd. Let $P_n$ and $C_n$ denote the path graph and cycle graph on $n$ vertices, respectively. 
Let $C_2^X$ denote the collection of all $2$-subsets of the set $X$. As usual $\binom{j}{k}=0$ if $j<k$. Let $E_n$ denote the complement graph of $K_n$. In the following theorem, we show some known independence numbers of the $2$-token graphs of some graphs that were proved in~\cite{alba} except for the case (\textit{vi}) which is a well-known~\cite{FFHH,johnson}.

\begin{theorem}\label{indepenk}
\begin{enumerate} Let $m$ be an integer
\item[i)] $\alpha(F_2(P_m))=\lfloor m^2/4\rfloor$, $m \geq 2$;
\item[ii)]  $\alpha(F_2(C_m))= \left\lfloor m \left\lfloor m/2 \right\rfloor/2\right\rfloor$, $m \geq 3$;
\item[iii)] $\alpha(F_2(E_m))=\binom{m}{2}$, $m \geq 2$;
\item[iv)] $ \alpha(F_2(K_{m,1}))=m$, for $m \in \{1,2\}$;
\item[v)] $ \alpha(F_2(K_{m,1}))=\binom{m}{2}$, for $m \geq 3$;
\item[vi)] $\alpha(F_2(K_m))=\left\lfloor m/2 \right\rfloor$, $m \geq 2$.
\end{enumerate}
\end{theorem}

Let $G=G_1+G_2$. Let $\bb_i=C_2^{V(G_i)}$, for $i=1, 2$. Let $\mathcal{B}=\bb_1 \cup \bb_2$ and $\rr=\{\{u, v\} \colon u \in V(G_1), v \in V(G_2)\}$. Note that $\{\bb_1, \bb_2, \rr\}$ is a partition of  $V(F_2(G))$ and that $F_2(G_i)$ is an induced subgraph of $F_2(G)$, where $V(F_2(G_i))=\bb_i$, for $i=1,2$. Therefore, we have the following lower bound for $\alpha(F_2(G))$.

\begin{proposition}\label{cotainfe}
Let $G=G_1+G_2$. Then $\alpha(F_2(G)) \geq \alpha(F_2(G_1))+ \alpha(F_2(G_2))$.
\end{proposition}
We define an independent set of $F_2(G_1+G_2)$ that will be useful to obtain an upper bound for $\alpha(F_2(E_n+G_2))$ as will be shown in Lemma~\ref{conjasociado}.

\begin{definition}\label{indeprs} 
\begin{itemize}
\item[]
\item[i)] Let $I_1$ and $I_2$ be independent sets in $G_1$ and $G_2$, respectively.  Let $\GG_i=F_2(G_i-I_i)$, for $i=1, 2$. Let $\ii_i$ be an independent set of $\GG_i$ such that $|\ii_i|=\alpha(\GG_i)$, for $i=1, 2$. Let $\ii_\bb=\ii_1 \cup \ii_2$, $\ii_\rr=\{\{u, v\} \colon u \in I_1, v \in I_2\}$ and $\ii_{I_1, I_2}=\ii_\rr \cup \ii_\bb$. Note that $\ii_{I_1, I_2}$ is an independent set in $F_2(G)$ and we call it an independent set associated with $I_1$  and $I_2$.
\item[ii)] Let $G=E_n+H$ and $\ii$ be an independent set of $F_2(G)$ such that $\ii \cap \rr \neq \emptyset$. Let $N_\ii(u)=\{v \in V(H) \colon \{u, v\} \in \ii\}$, for every $u \in V(E_n)$. Note that $N_\ii(u)$ is an independent set of $H$. Let $u'$ be any but fixed vertex in $V(E_n)$ such that $|N_\ii(u')| \geq |N_\ii(u)|$, for every $u \in V(E_n)$.  Let $S_1=\{u \in V(E_n) \colon N_\ii(u) \neq \emptyset\}$ and $S_2=N_\ii(u')$. We have that $S_1$ and $S_2$ are nonempty independent sets of $E_n$ and $H$, respectively. We define $\ii_{S_1, S_2}$ as the independent set associated with $S_1$ and $S_2$ as in part (i) of this definition, where $G_1=E_n$ and $G_2=H$.
\end{itemize}
\end{definition}
We have that $|\ii_{S_1, S_2}|=|S_1||S_2|+\alpha(\GG_1)+\alpha(\GG_2)$.

\begin{lemma}\label{conjasociado}
Let $G=E_n+H$. Let $\ii$ be an independent set of $F_2(G)$ such that $\ii \cap \rr \neq \emptyset$. Then there exists independent sets $S_1$ and $S_2$ of $E_n$ and  $H$, respectively, such that $|\ii_{S_1, S_2}| \geq |\ii|$. 
\end{lemma}
\begin{proof}
The existence of $S_1$ and $S_2$ is showed in Definition~\ref{indeprs}(\textit{ii}). Let $\ii_1, \ii_2, \ii_\bb$ and $\ii_\rr$ be as in Definition~\ref{indeprs}(i) with $I_1=S_1$ and $I_2=S_2$. As $\ii_\rr=\ii_{S_1,S_2} \cap \rr$ and $\ii_\bb=\ii_{S_1,S_2} \cap \bb$, then in order to show that $|\ii_{S_1,S_2}| \geq |\ii|$ it is enough to verify that $|\ii_\rr| \geq |\ii \cap \rr|$ and $|\ii_\bb| \geq |\ii \cap \bb|$. Let $s_1=|S_1|$. First note that 
\[
|\ii \cap \rr|=\sum_{u \in S_1} |N_\ii(u)|
\leq \sum_{u \in S_1} |N_\ii(u')|
=s_1|N_\ii(u')|
=|\ii_\rr|.
\]
Now, as $\ii \cap \bb_1$ and $\ii \cap \bb_2$ are independent sets in $\GG_1$ and $\GG_2$, respectively, we have that $|\ii  \cap \bb_1| \leq |\ii_1|$, $|\ii  \cap \bb_2| \leq |\ii_2|$. Therefore $|\ii  \cap \bb|\leq |\ii_1|+|\ii_2|=|\ii_\bb|$. 
\end{proof}
Let $S_1$, $S_2$ and $\GG_2$ be as in Definition~\ref{indeprs}(\textit{ii}). Then 
\begin{align}
|\ii_{S_1, S_2}|=|S_1||S_2|+\binom{n-|S_1|}{2}+\alpha(\GG_2).
\end{align}

\subsection{Independence number of $F_2(P_m-A)$}

The following result will be useful in the proofs of Theorems~\ref{tabanico} and \ref{trueda}. If $A$ is a proper subset of $P_m$, then every component of $P_m-A$ is isomorphic to $P_i$ for some $1\leq i \leq m$. The following result expresses $\alpha(F_2(P_m-A))$ in terms of the number of odd components of $P_m-A$.

\begin{theorem}\label{propcaminomenos} 
Let $G$ be a graph of order $m \geq 2$ consisting of $\kappa$ components such that each is isomorphic to $P_i$, for some $1\leq i \leq m$. Then $\alpha(F_2(G))=\frac{m^2+t^2-2t}{4}$, where $t$ is the number of odd components of $G$.  
\end{theorem}
\begin{proof}

Let $G_1,\dots,G_\kappa$ be the components of $G$ ordered in such a way that $|G_i|$ is odd for $0\leq i \leq t$ ($i=0$ means that $G$ has not odd components) and $|G_i|$ is even for $t <  i \leq \kappa$. Let $m_i=|G_i|$ for every $i$.  Let $V(G_i)=\{u_{i1},\dots,u_{im_i }\}$ and $E(G_i)=\{\{u_{ij},u_{i(j+1)}\} ~| 1\leq j \leq m_i-1\}$. Now we define an independent set $I$ of $F_2(G)$. Let  
\[
U=\{\{u_{il},u_{jk}\} \in V(F_2(G)) \mid 1\leq i<j\leq \kappa, \text{ where } l \text{ and } k \text{ have the same parity}\}.
\]
Let $V_i=\{\{u_{il}, u_{ik }\} \in V(F_2(G)) \mid \text{ where } l \text{ and } k \text{ have different parity} \}$ and $V=\bigcup\limits_{i=1}^\kappa V_i$.  Let $I=U \cup V$. It is easy to check that $I$ is an independent set of $F_2(G)$. As $U$ and $V$ are disjoint sets, $|I|=|U|+|V|$. Note first that
\[
|V| =\sum\limits_{i=1}^\kappa\left\lceil\frac{m_i}{2}\right\rceil\left\lfloor\frac{m_i}{2}\right\rfloor
 =\sum\limits_{i=1}^\kappa\left\lfloor\frac{m_i^2}{4}\right\rfloor
 =\sum\limits_{i=1}^t\frac{m_i^2-1}{4}+\sum\limits_{i=t+1}^\kappa\frac{m_i^2}{4}
 =\left(\sum\limits_{i=1}^\kappa\frac{m_i^2}{4}\right)-\frac{t}{4}.
\]

To calculate $|U|$, we run over all components of $G$. 

\begin{align*}
|U| & =\sum_{i=1}^{\kappa-1}\left( \left\lceil\frac{m_i}{2}\right\rceil \sum_{j=i+1}^{\kappa} \left\lceil\frac{m_j}{2}\right\rceil  + \left\lfloor\frac{m_i}{2}\right\rfloor \sum_{j=i+1}^{\kappa}\left\lfloor\frac{m_j}{2}\right\rfloor   \right)\\
& =\sum_{i=1}^{\kappa-1}\left( \left\lceil\frac{m_i}{2}\right\rceil \left( \sum_{j=i+1}^{\kappa} \frac{m_j}{2}+ \sum_{j=i+1}^{t} \frac{1}{2} \right)  + \left\lfloor\frac{m_i}{2}\right\rfloor \left(  \sum_{j=i+1}^{\kappa} \frac{m_j}{2}- \sum_{j=i+1}^{t} \frac{1}{2} \right)   \right)\\
&=\sum_{i=1}^{\kappa-1} \left(   \left( \left\lceil\frac{m_i}{2}\right\rceil + \left\lfloor\frac{m_i}{2}\right\rfloor \right) \sum_{j=i+1}^{\kappa} \frac{m_j}{2} + \left( \left\lceil\frac{m_i}{2}\right\rceil - \left\lfloor\frac{m_i}{2}\right\rfloor \right)\sum_{j=i+1}^{t} \frac{1}{2}  \right)\\
&=\sum_{i=1}^{\kappa-1} \left(   m_i \sum_{j=i+1}^{\kappa} \frac{m_j}{2} + (m_i \bmod 2 )\sum_{j=i+1}^{t} \frac{1}{2}  \right)\\
&=\sum_{i=1}^{\kappa-1}   m_i \sum_{j=i+1}^{\kappa} \frac{m_j}{2}+ \sum_{i=1}^{t} \sum_{j=i+1}^{t} \frac{1}{2}  \\
&=\sum_{i=1}^{\kappa-1}   m_i \sum_{j=i+1}^{\kappa} \frac{m_j}{2}+ \frac{t^2-t}{4}.
\end{align*}

Then
\begin{align*}
|I|=|U|+|V| & =2\sum_{i=1}^{\kappa-1}   \frac{m_i}{2} \sum_{j=i+1}^{\kappa} \frac{m_j}{2}  + \sum\limits_{i=1}^\kappa\frac{m_i^2}{4}+\frac{t^2-2t}{4}
& = \left(\sum_{i=1}^{\kappa} \frac{m_i}{2} \right)^2+  \frac{t^2-2t}{4}\\
& =\frac{m^2+t^2-2t}{4}.
\end{align*}
 Therefore, $\alpha(F_2(G))\geq \frac{m^2+t^2-2t}{4}$. Now, let $J$ be an independent set of $F_2(G)$. Let $D=\{\{u_{il},u_{jk}\} \mid 1\leq i<j\leq \kappa\}$. Note that $J$ is the disjoint union of two sets $J'$ and $J''$, where $J'=\cup_{i=1}^\kappa J'_i$, with $J'_i \subset C_2^{V(G_i)}$, for every $i$, and  $J'' \subset D$. For a vertex $\{u_{il},u_{ik}\} \in J'_i$ we assume, without loss of generality, that $l<k$. Let $Q_i'$ be the set obtained  from $J_i'$ by changing all the vertices  $\{u_{il},u_{ik}\}$ in $J'_i$ such that $l$ and $k$ have the same parity to $\{u_{il},u_{ik-1}\}$. In this way, $Q_i' \subset V_i$ and $|J_i'|=|Q_i'| \leq |V_i|$, for every $i$. For a vertex $\{u_{il},u_{jk}\} \in J''$ we assume, without loss of generality, that $i<j$. We construct a set $Q''$ from $J''$ by changing all the vertices  $\{u_{il},u_{jk}\}$ in $J''$ such that $l$ and $k$ have different parity in the following way: if $l$ is even we change $\{u_{il},u_{jk}\}$ to $\{u_{il-1},u_{jk}\}$ and if $l$ is odd we change $\{u_{il},u_{jk}\}$ to $\{u_{il},u_{jk-1}\}$. In this way $Q'' \subset U$ and $|J''|=|Q''|\leq |U|$. Therefore $|J|=|J'|+|J''| \leq |V|+|U|=|I|=\frac{m^2+t^2-2t}{4}$.
\end{proof}

\section{Fan graphs}\label{sec3}
In this section, we obtain the independence number of the $2$-token graphs of fan graphs.  
\begin{theorem}\label{tabanico} Let $m\geq1$ and $n \geq 1$ be integers. Let $F_{n,m}=E_n+P_m$. Then 
\[ 
\alpha(F_2(F_{n,m}))=\begin{cases} 
      n\left\lceil\frac{m}{2}\right\rceil+\binom{\lfloor\frac{m}{2}\rfloor}{2} & n\in\{\frac{m+1}{2},\frac{m+3}{2}\}\\
\left\lfloor\frac{m^2}{4}\right\rfloor+\binom{n}{2} & \text{other case. }\\
  \end{cases}
\]

\end{theorem}

\begin{proof}
When $m=1$, $F_{n, 1} \simeq K_{n, 1}$. In~\cite{alba} showed that $\alpha(F_2(K_{n,1}))=n$, for $n \in \{1, 2\}$ and $\alpha(F_2(K_{n,1}))=\binom{n}{2}$, for $n \geq 3$. Therefore, we can assume that $m \geq 2$. 

Let $G=F_{n,m}$. By Proposition~\ref{cotainfe} and parts (\textit{i}) and (\textit{iii}) of Theorem~\ref{indepenk} we have that $\alpha(F_2(G))\geq \left\lfloor\frac{m^2}{4}\right\rfloor+\binom{n}{2}$. 

Let $\ii$ be an independent set of $F_2(G)$. If $\ii \cap \rr =\emptyset$, then $\ii \subseteq \bb$ and hence $|\ii| \leq \left\lfloor\frac{m^2}{4}\right\rfloor+\binom{n}{2}$. If $\ii \cap \rr \neq \emptyset$, then by Lemma~\ref{conjasociado} it is enough to show that $ |\ii_{S_1, S_2}| \leq \left\lfloor\frac{m^2}{4}\right\rfloor+\binom{n}{2}$, where $\ii_{S_1, S_2}$ is an independent set of $F_2(G)$ associated with $\ii$ as in Definition~\ref{indeprs}(\textit{ii}), with $H=P_m$. Let  $r=|S_1|$ and $s=|S_2|$. Then 
$|\ii_{S_1, S_2}|  =rs+\binom{n-r}{2} +\alpha(\GG_2).$
Where $\GG_2=F_2(P_m-S_2)$. By Theorem~\ref{propcaminomenos} we have that
\begin{align}\label{cardinalisr}
|\ii_{S_1, S_2}| =rs+\binom{n-r}{2} + \frac{(m-s)^2+t^2-2t}{4},
\end{align}
where $t$ is the number of odd components of $P_m-S_2$. Let $\ka$ be the number of components of $P_m-S_2$. As $S_2$ is an independent set of $P_m$, then $\ka \in \{s-1, s, s+1\}$ depending on the number of end vertices of $P_m$ in $S_2$. 

We will show that $\left\lfloor\frac{m^2}{4}\right\rfloor+\binom{n}{2}-|\ii_{S_1, S_2}|\geq0$, except when $n \in \{\frac{m+1}{2},\frac{m+3}{2}\}$. First note that
\[
\left\lfloor\frac{m^2}{4}\right\rfloor=\begin{cases} 
      \frac{m^2}{4}  & m \text{ even }\\
     \frac{m^2}{4}-\frac{1}{4}& m \text{ odd }.\\
  \end{cases}
  \]
Let $n_1:=n-r$ and $m_1:=m-2s$. The proofs of the following inequalities are by direct calculations.

\begin{claim}\label{propaux}

\begin{enumerate} Let $c=0$ when $m$ is even and $c=-\frac{1}{4}$ when $m$ is odd. 
\item[i)] If $t\leq s-1$, then

\begin{align*}
\left\lfloor\frac{m^2}{4}\right\rfloor+\binom{n}{2}-|\ii_{S_1,S_2}| &\geq\frac{(r-s)(r-s-1)}{2}+\frac{s(m_1+1)+2rn_1}{2}-\frac{3}{4}+c. 
\end{align*}
\item[ii)] If $t \leq s$, then
\begin{align*}
\left\lfloor\frac{m^2}{4}\right\rfloor+\binom{n}{2}-|\ii_{S_1,S_2}| &\geq \frac{(r-s)(r-s-1)}{2}+\frac{sm_1+2rn_1}{2}+c. 
\end{align*}
\item[iii)] If $t \leq s+1$, then
\begin{align*}
\left\lfloor\frac{m^2}{4}\right\rfloor+\binom{n}{2}-|\ii_{S_1,S_2}| &\geq \frac{(r-s)^2}{2}+\frac{2rn_1-(r-sm_1)}{2}+\frac{1}{4}+c. 
\end{align*}

\end{enumerate}

\end{claim}
We proceed by cases depending on the parity of $m$. 
\begin{itemize}

\item[Case 1.] Suppose that $m$ is even. Then $1\leq s\leq \frac{m}{2}$ and $1\leq r\leq n$. Even more $n\not\in\{\frac{m+1}{2},\frac{m+3}{2}\}$. 

\begin{enumerate}
\item[Case 1.1.] Suppose that $s=\frac{m}{2}$. In this case $\ka \neq s+1$, and hence $t\leq \ka \leq s$. By Claim~\ref{propaux}(\textit{ii}), with $m_1=0$ and $c=0$, we have
\small{
\begin{align*}
\left\lfloor\frac{m^2}{4}\right\rfloor+\binom{n}{2}-|\ii_{S_1,S_2}|
& \geq \frac{(r-s)(r-s-1)}{2}+\frac{2rn_1}{2}\geq 0, 
\end{align*}}
because $r, s \geq 1$ and $n_1\geq 0$.

\item[Case 1.2.] Suppose that $s<\frac{m}{2}$. We don't know the exact value for $\ka$ or $t$, but it is sufficient to consider the maximum possible value for $t$, equal to $s+1$. Since $s<\frac{m}{2}$ and $m$ is even, then $m_1 \geq 2$. By Claim~\ref{propaux}(\textit{iii}), with $c=0$ and $m_1\geq 2$, we have

\small{
\begin{align*}
\left\lfloor\frac{m^2}{4}\right\rfloor+\binom{n}{2}-|\ii_{S_1,S_2}| & \geq \frac{(r-s)^2}{2}+\frac{2rn_1-(r-2s)}{2}+\frac{1}{4}\geq \frac{(r-s)(r-s-1)}{2}\geq 0,
\end{align*}}
where the last inequality follows because $s, r\geq1$ and $n_1\geq 0$.

\end{enumerate}
\item[Case 2.]  Suppose that $m$ is odd. In this case $1\leq s\leq \frac{m+1}{2}$.

\begin{enumerate}

\item[Case 2.1.] Suppose that $s<\frac{m+1}{2}$. As $m\geq 3$, then $m_1 \geq 1$. It is enough to consider $t\leq s+1$. By Claim~\ref{propaux}(\textit{iii}), with $c=-\frac{1}{4}$ and $m_1\geq 1$, we have

\small{
\begin{align*}
\left\lfloor\frac{m^2}{4}\right\rfloor+\binom{n}{2}-|\ii_{S_1,S_2}|&\geq \frac{(r-s)^2}{2}+\frac{2rn_1-(r-s)}{2}\geq \frac{(r-s)(r-s-1)}{2}\geq 0,
\end{align*}}
because $r, s \geq 1$ and $n_1\geq 0$.

\item[Case 2.2.] Suppose that $s=\frac{m+1}{2}$ and $ r<n$. In this case $t=\ka=s-1$ because $|G_i|=1$, where $G_i$ is a component of $P_m-S_2$, for every $i=1,\dots,\ka$. We have that $s\geq 2$ and $n_1\geq 1$ because $m\geq3$ and $n > r$, respectively. By Claim~\ref{propaux}(\textit{i}), with $c=-\frac{1}{4}$ and $m_1= -1$, we have

\small{
\begin{align*}
\left\lfloor\frac{m^2}{4}\right\rfloor+\binom{n}{2}-|\ii_{S_1,S_2}|&\geq\frac{(r-s)(r-s-1)}{2}+\frac{2n_1 r}{2}-1 \geq0,
\end{align*}}

because $s\geq2$, $r \geq 1$ and $n_1\geq 1$.

\item[Case 2.3.] Suppose that $s=\frac{m+1}{2}$ and $r=n$. Similarly as in the previous case, $t=s-1$. Let  $\gamma:=n-s$. By Claim~\ref{propaux}(\textit{i}), with $c=-\frac{1}{4}$, $n_1=0$ and $m_1= -1$, we have
\small{
\begin{align*}
\left\lfloor\frac{m^2}{4}\right\rfloor+\binom{n}{2}-|\ii_{S_1,S_2}|&\geq\frac{(n-s)(n-s-1)}{2}-1=\frac{\gamma^2-\gamma}{2}-1.
\end{align*}}

Note that $\frac{\gamma^2-\gamma}{2}-1\geq 0$ if and only if $\gamma \not \in\{0,1\}$. That is, $|\ii_{S_1,S_2}| \leq \left\lfloor\frac{m^2}{4}\right\rfloor+\binom{n}{2}$, for $n\not\in\{\frac{m+1}{2}, \frac{m+3}{2}\}$. Therefore $\alpha(F_2(G))\leq \left\lfloor\frac{m^2}{4}\right\rfloor+\binom{n}{2}$, for $n \not \in\{\frac{m+1}{2},\frac{m+3}{2}\}$. 

The reminder case is when $n  \in\{\frac{m+1}{2},\frac{m+3}{2}\}$. As $s=\frac{m+1}{2}$, that is equal to $\alpha(P_m)$, and $r=n$, that is equal to $\alpha(E_n)$, then $\ii$ is a maximum independent set of $F_2(G)$ with the property that $\ii \cap \rr \neq \emptyset$. By Lemma~\ref{conjasociado}, $|\ii_{S_1,S_2}| \geq \ii $ and hence $|\ii_{S_1,S_2}|$ is an upper bound for $\alpha(F_2(G))$. In fact, the set $\ii_{S_1,S_2}$ also provides a lower bound for $\alpha(F_2(F_{n,m}))$. We have that  
\small{
\begin{align*}
|\ii_{S_1,S_2}|
&=rs+\frac{(m-s)^2}{4}+\frac{t^2-2t}{4}=n\left\lceil\frac{m}{2}\right\rceil+\binom{\left\lfloor\frac{m}{2}\right\rfloor}{2},
\end{align*}}   

because $s=\frac{m+1}{2}=\lceil\frac{m}{2}\rceil$ ($m$ is odd), $r=n$ and $t=s-1$. Therefore
$
\alpha(F_2(F_{n,m}))=\left\lceil m/2 \right\rceil+\binom{\lfloor m/2 \rfloor}{2},
$
 for $n\in\{\frac{m+1}{2},\frac{m+3}{2}\}$.
 \end{enumerate}
 \end{itemize}
\end{proof}

\section{Wheel graphs}\label{sec4}
In this section, we obtain the independence number of the $2$-token graph of wheel graphs. 
Let $W_{n,m}=E_n+C_m$. It is easy to check that $\alpha(F_2(W_{3,1}))=2$ and $\alpha(F_2(W_{3,2}))=3$.

\begin{theorem}\label{trueda} Let $m\geq 3$ and $n\geq 1$. Then
  \[
  \alpha(F_2(W_{n,m}))=\left\lfloor\frac{m}{2}\left\lfloor\frac{m}{2}\right\rfloor\right\rfloor + \binom{n}{2},
  \]
  except when $(m, n) \in \{(3, 1), (3, 2)\}$.
\end{theorem}
\begin{proof}
First, by Proposition~\ref{cotainfe} and parts (\textit{ii}) and (\textit{iii}) of Theorem~\ref{indepenk} it follows that $\alpha(F_2(W_{n,m}))\geq \left\lfloor m \left\lfloor m/2 \right\rfloor/2\right\rfloor + \binom{n}{2}$. Let $\ii$ be an independent set of $F_2(W_{n,m})$. If $\ii \cap \rr=\emptyset$ we are done. Suppose now that $\ii \cap \rr \neq \emptyset$. Let $S_1$, $S_2$ and $\ii_{S_1,S_2}$ be as in Definition~\ref{indeprs}(\textit{ii}). Let $r=|S_1|$ and $s=|S_2|$. The number $\alpha(\GG_2)$ is given by Proposition~\ref{propcaminomenos} because $C_m-S_2$ is a graph where each component is isomorphic to some path graph. Therefore 

\[
|\ii_{S_1,S_2}| =rs+\binom{n-r}{2} +\frac{(m-s)^2+t^2-2t}{4}.
\]

Note that this formula is the same as that for the case of $F_{n,m}$. Therefore, when $m$ is even the result follows by Case 1 in the proof of  theorem~\ref{tabanico} (because $\left\lfloor m \left\lfloor m/2 \right\rfloor/2\right\rfloor = \left\lfloor m^2/4 \right\rfloor$). Suppose now that $m$ is odd. We will continue by cases. 

\begin{itemize}
\item[Case 1.] Suppose that $m=3$. In this case $n \geq 3$, $\left\lfloor m \left\lfloor m/2 \right\rfloor/2\right\rfloor =1$, $s=1$ and $t=0$. Therefore $|\ii_{S_1,S_2}|=r+\binom{n-r}{2}+1$ and hence 

\[
\left\lfloor\frac{m}{2}\left\lfloor\frac{m}{2}\right\rfloor\right\rfloor+\binom{n}{2}-|\ii_{S_1,S_2}| = \binom{n}{2}-\binom{n-r}{2}-r\geq 0,
\]
because $n\geq 3$ and $1\leq r\leq n$.

\item[Case 2.] Suppose that m=5. In this case $\left\lfloor m \left\lfloor m/2 \right\rfloor/2\right\rfloor =5$ and $s\in\{1, 2\}$. If $s=1$, then $t=0$ and $\alpha(\GG_2)=4$. Therefore
\[
\left\lfloor\frac{m}{2}\left\lfloor\frac{m}{2}\right\rfloor\right\rfloor+\binom{n}{2}-|\ii_{S_1,S_2}| = \binom{n}{2}-\binom{n-r}{2}+1-r\geq 0, 
\]
because  $1\leq r\leq n$. If $s=2$, then $t=1$ and $\alpha(\GG_2)=2$. Therefore
\[
\left\lfloor\frac{m}{2}\left\lfloor\frac{m}{2}\right\rfloor\right\rfloor+\binom{n}{2}-|\ii_{S_1,S_2}|= \binom{n}{2}-\binom{n-r}{2}+3-2r\geq 0, 
\]
because $1\leq r\leq n$. 

\item[Case 3.] Suppose that $m \geq 7$. When $m$ is odd 

\[
\left\lfloor\frac{m}{2}\left\lfloor\frac{m}{2}\right\rfloor\right\rfloor=\begin{cases} 
      \frac{m^2-m}{4}  & m \equiv 1 \bmod 4\\
     \frac{m^2-m}{4}-\frac{1}{2}&m \equiv 3 \bmod 4.\\
  \end{cases}
  \]
So that it is enough to show that $(m^2-m)/4-1/2+\binom{n}{2}-|\ii_{S_1,S_2}|\geq0$, for every $m \geq 7$ odd. 
As $S_2$ is an independent set of $C_m$, the graph $C_m-S_2$ has $s$ components, with $1\leq s\leq \frac{m-1}{2}$. As $m$ is odd,  $C_m-S_2$ has at least one even component. Therefore $1\leq t\leq s-1$.  Let $n_1:=n-r$ and $m_1:=m-2s$. Note that
\small{
\begin{align}
\frac{m^2-m}{4}-\frac{1}{2}+\binom{n}{2}-|\ii_{S_1,S_2}|&\geq\frac{(r-s)(r-s-1)}{2}+\frac{m_1(2s-1)+4n_1 r}{4}-\frac{5}{4}. \label{ecuam=5}
\end{align}}
We continue by cases.
\begin{enumerate}
\item[Case 3.1.] Suppose that $s=\frac{m-1}{2}$. Then $m_1=1$. As $m\geq 7$, we have that $s\geq 3$. In this case $t=s-1$. Substituting these values in equation~(\ref{ecuam=5}) we obtain
\[
\left\lfloor\frac{m}{2}\left\lfloor\frac{m}{2}\right\rfloor\right\rfloor+\binom{n}{2}-|\ii_{S_1,S_2}|
\geq \frac{(r-s)(r-s-1)}{2}+\frac{(2s-1)+4n_1 r}{4}-\frac{5}{4} \geq 0,
\]
because $s \geq 3$, $r\geq 1$ and $n_1\geq 0$.
\item[Case 3.2.] Suppose that $s<\frac{m-1}{2}$. If $s=1$, then $m_1 \geq 5$ because $m \geq 7$. By equation~(\ref{ecuam=5}) we obtain
\[
\left\lfloor\frac{m}{2}\left\lfloor\frac{m}{2}\right\rfloor\right\rfloor+\binom{n}{2}-|\ii_{S_1,S_2}|
\geq \frac{(r-1)(r-2)}{2}+\frac{5+4\alpha r}{4}-\frac{5}{4} \geq 0.
\]
because $r\geq 1$ and $n_1\geq 0$. Now, if $s\geq 2$, then $m_1 \geq 3$ because $m \geq 7$. Then
\[
\left\lfloor\frac{m}{2}\left\lfloor\frac{m}{2}\right\rfloor\right\rfloor+\binom{n}{2}-|\ii_{S_1,S_2}|
\geq \frac{(r-s)(r-s-1)}{2}+\frac{9+4n_1 r}{4}-\frac{5}{4} \geq 0.
\]
because $r\geq 1$,  $s\geq 2$ and $n_1\geq 0$.
\end{enumerate}

\end{itemize}

\end{proof}

\section{Graph $E_n+K_m$}\label{sec5}

Let  $G=E_n+K_m$. If $n=1$, then $G \simeq K_{m+1}$ and $\alpha(F_2(G))=\floor{(m+1)/2}$ by Theorem~\ref{indepenk}(\textit{vi}) . For $n\geq 2$ we have the following result

\begin{theorem}\label{teokmen}
If $m\geq 1$ and $n \geq 2$ are integers, then
\[ 
\alpha(F_2(G))=\begin{cases} 
      \ceil{\frac{m+2}{2}} & n=2\\
     \floor{\frac{m}{2}}+\binom{n}{2}& n \geq 3.\\
  \end{cases}
\]
\end{theorem}

\begin{proof}

First, we consider the case $n=2$. Let $\ii$ be an independent set of $F_2(G)$. If $\ii\cap\mathcal{R}=\emptyset$, by parts  (\textit{iii}) and (\textit{vi}) of Theorem~\ref{indepenk} we have that
 \[
|\ii|\leq \alpha(F_2(E_2))+\alpha(F_2(K_m)=1+ \left\lfloor \frac{m}{2} \right\rfloor \leq  \left\lceil \frac{m+2}{2} \right\rceil.
\]
Suppose now that $\ii\cap\mathcal{R} \neq \emptyset$. Let $\ii_{S_1, S_2}$ be an independent set associated with $\ii$ as in Definition~\ref{indeprs}(\textit{ii}). Let $r=|S_1|$ and $s=|S_2|$. In this case $1\leq r \leq 2$ and $s=1$. Therefore 
\[
|\ii_{S_1, S_2}|=r+\binom{n-r}{2}+\alpha(F_2(K_{m-1})=r+\left\lfloor \frac{m-1}{2} \right\rfloor
\]
and hence
\[
\left\lceil\frac{m+2}{2}\right\rceil-|\ii_{S_1, S_2}|=\left\lceil\frac{m+2}{2}\right\rceil-\left\lfloor \frac{m-1}{2}\right\rfloor -r=2-r\geq 0.
\]
as desired.  Finally, if $m$ is even, let  
\[
I=\{\{v_1, v_2\},\{v_3, v_4\},\dots, \{v_{m-1}, v_m\}\} \cup \{\{u_1, u_2\}\},
\]
and if $m$ is odd let
\[
I=\{\{v_1, v_2\},\{v_3, v_4\},\dots, \{v_{m-2}, v_{m-1}\}\} \cup \{ \{u_1, v_m\},\{u_2, v_m\} \}.
\]
 In both cases $I$ is an independent set of $\alpha(F_2(G)$ of cardinality $\lceil(m+2)/2\rceil$. 

Now we work with the case $n \geq 3$. By Proposition~\ref{cotainfe} and parts  (\textit{iii}) and (\textit{vi}) of Theorem~\ref{indepenk} it is enough to show that $\alpha(F_2(G))\leq\lfloor\frac{m}{2}\rfloor+\binom{n}{2}$. Let $\ii$ be and independent set of $F_2(G)$. If $\ii\cap\mathcal{R}=\emptyset$ we are done. Suppose that $\ii\cap\mathcal{R}\not=\emptyset$. By Lemma~\ref{conjasociado} it follows that is enough to show that $|\ii_{S_1, S_2}|\leq\lfloor\frac{m}{2}\rfloor + \binom{n}{2}$, where $\ii_{S_1, S_2}$ is and independent set associated with $\ii$ as in Definition~\ref{indeprs}, where $S_1$ and $S_2$ are as in the proof of Lemma~\ref{conjasociado}, with $G'=K_m$. Let $r=|S_1|$ and $s=|S_2|$. In this case  $s=1$. Therefore 
\[
|\ii_{S_1,S_2}|=r+\binom{n-r}{2}+ \left\lfloor\frac{m-1}{2}\right\rfloor.
\] 
If $m$ is even, then
\small{
\begin{align*}
\left\lfloor\frac{m}{2}\right\rfloor+\binom{n}{2}-r-\left\lfloor\frac{m-1}{2}\right\rfloor-\binom{n-r}{2}&=\binom{n}{2}-\binom{n-r}{2}-r+1 \geq 0,
\end{align*}}
because $n \geq 3$ and $1\leq r \leq n$. 

If $m$ is odd, then \begin{align*}
\left\lfloor\frac{m}{2}\right\rfloor+\binom{n}{2}-r-\left\lfloor\frac{m-1}{2}\right\rfloor-\binom{n-r}{2}&=\binom{n}{2}-\binom{n-r}{2}-r \geq 0,
\end{align*}
because $n \geq 3$ and $1\leq r \leq n$.

\end{proof}

\section{Complete bipartite graph}\label{sec6}
In this section, we give another proof of a Theorem of de Alba et al.,~\cite{alba}  about the independence number of $F_2(K_{n,m})$. Remember that $K_{n,m}=E_n+E_m$. 
\begin{theorem}
Let $m,n\geq1$ be integers. Then $\alpha(F_2(K_{n,m}))=\max\{mn,\binom{m}{2}+\binom{n}{2}\}$.
\end{theorem}
\begin{proof}
Note that in this case $\mathcal{B}$ and $\mathcal{R}$ are disjoint independent sets of $F_2(K_{n,m})$, with $|\mathcal{B}|=\binom{m}{2}+\binom{n}{2}$ and $|\mathcal{R}|=mn$. Therefore $\beta(F_2(K_{n,m}))\geq\max\{mn,\binom{m}{2}+\binom{n}{2}\}$.

Let $\ii$ be an independent set of $F_2(K_{n,m})$. If $\ii\cap\mathcal{R}=\emptyset$ we are done because $|\ii|\leq\binom{m}{2}+\binom{n}{2}\leq\max\{mn, \binom{m}{2}+\binom{n}{2}\}$. Now, suppose that $\ii\cap\mathcal{R}\not=\emptyset$. By Lemma~\ref{conjasociado} it is enough to show that $|\ii_{S_1,S_2}|\leq\max\{mn,\binom{m}{2}+\binom{n}{2}\}$, where $\ii_{S_1,S_2}$ is as in Definition~\ref{indeprs}(\textit{ii}). Let $r=|S_1|$ and $s=|S_2|$. Therefore
\[
|\ii_{S_1,S_2}|=rs+\binom{n-r}{2}+\binom{m-s}{2}.
\]
If $n=r$ and $m=s$, then $|\ii_{S_1,S_2}|=mn$ and hence $\max\{mn,\binom{m}{2}+\binom{n}{2}\}-|\ii_{S_1,S_2}|\geq0$. Finally, it is an exercise to show that if $n \neq r$ or $m \neq s$, then 
\small{
\begin{align*}
\binom{m}{2}+\binom{n}{2}-|\ii_{S_1,S_2}|&=\binom{m}{2}-\binom{m-s}{2}+\binom{n}{2}-\binom{n-r}{2}-rs \geq 0.
\end{align*}}
because $1\leq r < n$ and $1 \leq s <m$. 
\end{proof}







\bibliography{independence-arxiv}

\end{document}